\documentclass[12pt]{amsart}

\usepackage{amsfonts, amssymb, amsmath, amsthm}
\usepackage[latin1]{inputenc}
\usepackage{amsmath}
\usepackage{amssymb}
\usepackage{amsfonts}
\usepackage{latexsym}
\usepackage{mathtools}
\usepackage{amscd}
\usepackage{xy} \xyoption{all}
\usepackage{enumerate}
\usepackage{mathtools}
\usepackage{color}
\newcommand{\annotation}[1]{}

\newcommand{\aut}{\textnormal{Aut}}
\newcommand{\autinf}[1]{\textnormal{Aut}^{(\infty)}(\sigma_{#1})}
\newcommand{\autk}[2]{\textnormal{Aut}^{(#1)}(\sigma_{#2})}

\newcommand{\inv}[1]{\textnormal{Inv}(\sigma_{n})}
\newcommand{\infinv}[1]{\textnormal{Inv}^{\infty}(\sigma_{n})}
\newcommand{\simpinf}[1]{\textnormal{Simp}^{(\infty)}(\Gamma_{#1})}
\newcommand{\simpaltinf}[1]{\textnormal{Simp}_{\textnormal{ev}}^{(\infty)}(\Gamma_{#1})}

\newcommand{\sym}{\textnormal{Sym}}
\newcommand{\simp}{\textnormal{Simp}}
\newcommand{\alt}{\textnormal{Alt}}
\newcommand{\ev}{\textnormal{ev}}
\DeclarePairedDelimiter\floor{\lfloor}{\rfloor}

\newtheorem*{theorem*}{Theorem}
\newtheorem{theorem}{Theorem}

\newtheorem{corollary}[theorem]{Corollary}
\newtheorem{proposition}[theorem]{Proposition}
\newtheorem{lemma}[theorem]{Lemma} 
\theoremstyle{remark}
\newtheorem{remark}[theorem]{Remark}
\theoremstyle{definition}
\newtheorem{definition}[theorem]{Definition}
\newtheorem*{definition*}{Definition}

\newtheorem*{notation*}{Notation}
\newtheorem{example}[theorem]{Example}
\input xy
\xyoption{all}

\author{Scott Schmieding}
\address{University of Denver, Denver, CO 80208 USA}
\email{scott.schmieding@du.edu}

\begin{document}

\title{Local $\mathcal{P}$ entropy and stabilized automorphism groups of subshifts}

\begin{abstract}
For a homeomorphism $T \colon X \to X$ of a compact metric space $X$, the stabilized automorphism group $\aut^{(\infty)}(T)$ consists of all self-homeomorphisms of $X$ which commute with some power of $T$. Motivated by the study of these groups in the context of shifts of finite type, we introduce a certain entropy for groups called local $\mathcal{P}$ entropy. We show that when $(X,T)$ is a non-trivial mixing shift of finite type, the local $\mathcal{P}$ entropy of the group $\aut^{(\infty)}(T)$ is determined by the topological entropy of $(X,T)$. We use this to give a complete classification of the isomorphism type of the stabilized automorphism groups of full shifts.
\end{abstract}

\maketitle
\tableofcontents

\section{Introduction}
By a dynamical system $(X,T)$ we mean a self-homeomorphism $T$ of a compact metric space $X$. An automorphism of a dynamical system $(X,T)$ is a homeomorphism $\phi \colon X \to X$ such that $\phi T = T \phi$, and the collection of all homeomorphisms of $(X,T)$ forms a group (under composition) which we denote by $\aut(T)$. \\
\indent When $(X,\sigma_{X})$ is a non-trivial mixing shift of finite type, the automorphism group $\aut(\sigma_{X})$ is an infinite, countable, residually finite, non-amenable group (see \cite{BLR88}). It is known to contain a wide variety of subgroups, including isomorphic copies of every finite group, as well as free groups \cite{BLR88}. First investigated in the case of full shifts by Hedlund and others in the late 60's, these groups have since been heavily studied (see \cite[Sec. 7]{BSstable} for a brief survey). Despite this, much of the structure of $\aut(\sigma_{X})$ remains a mystery. Exemplifying this, a fundamental question that remains open is the following: if $(X,\sigma_{X})$ and $(Y,\sigma_{Y})$ are mixing shifts of finite type with $\aut(\sigma_{X})$ isomorphic as a group to $\aut(\sigma_{Y})$, must $(X,\sigma_{X})$ be topologically conjugate to either $(Y,\sigma_{Y})$ or $(Y,\sigma_{Y}^{-1})$?\footnote{This appears as Question 4.1 in \cite{BLR88}.} A central case of this problem is when $(X,\sigma_{X}) = (X_{n},\sigma_{n})$ is the full shift on $n$ symbols, where the question can be rephrased as follows: for which $m$ and $n$ are $\aut(\sigma_{m})$ and $\aut(\sigma_{n})$ isomorphic as groups? The only non-trivial results addressing this problem are elementary consequences of Ryan's Theorem \cite{Ryan1,Ryan2}, as observed in \cite[Sec. 4]{BLR88}, where it is noted that (among other cases) the groups $\aut(\sigma_{2})$ and $\aut(\sigma_{4})$ can not be isomorphic. These techniques essentially exploit the structure of algebraic roots of the shift map, together with Ryan's Theorem identifying the center of $\aut(\sigma_{n})$ as the subgroup generated by $\sigma_{n}$. However, the question of whether, for example, $\aut(\sigma_{2})$ and $\aut(\sigma_{3})$ are isomorphic or not remains a particularly intriguing problem.\\
\indent In \cite{HKS}, a certain stabilization of the automorphism group of a system was introduced. Given $(X,T)$, one defines the stabilized automorphism group to be
$$\aut^{(\infty)}(T) = \bigcup_{n=1}^{\infty}\aut(T^{n})$$
where the union is taken in the group of all self-homeomorphisms of $X$. A study of the group $\aut^{(\infty)}(\sigma_{X})$ for $(X,\sigma_{X})$ a shift of finite type was undertaken in \cite{HKS}, where the following was proved:
\begin{theorem*}[\cite{HKS}]
Let $(X,\sigma_{X})$ be a non-trivial mixing shift of finite type.
The dimension representation\footnote{For background on the dimension representation, see either \cite{BLR88} or \cite{HKS}; the dimension representation does not play a role in the present paper.} $\pi_{X}$ of $\aut(\sigma_{X})$ extends to a stabilized dimension representation $\pi^{(\infty)}_{X}$ of $\aut^{(\infty)}(\sigma_{X})$, and the abelianization of $\aut^{(\infty)}(\sigma_{X})$ factors through $\pi^{(\infty)}_{X}$. If $(X_{n},\sigma_{n})$ is a full shift then $\pi_{X_{n}}^{(\infty)}$ maps onto a free abelian group of rank $\omega(n)$ (where $\omega(n)$ is the number of distinct prime divisors of $n$) and coincides with the abelianization map for $\autinf{n}$; in particular, if $\autinf{m}$ is isomorphic to $\autinf{n}$, then $\omega(m) = \omega(n)$. Moreover, for a full shift $(X_{n},\sigma_{n})$, the kernel of $\pi^{(\infty)}_{X_{n}}$ is simple.
\end{theorem*}
This result distinguishes, up to isomorphism, various stabilized automorphism groups using their abelianizations, but leaves open many cases: for example $\autinf{m}$ and $\autinf{n}$ when $\omega(m) = \omega(n)$. In particular, the question of whether e.g. $\autinf{2}$ and $\autinf{3}$ are isomorphic remained out of reach. To give some indication about the scope of the problem, note that for $m,n$ with $\omega(m) = \omega(n)$, the groups $\autinf{m}$ and $\autinf{n}$ have isomorphic abelianizations, and are each extensions of their abelianizations by some countable infinite simple group.\\
\indent Motivated by the study of these stabilized automorphism groups, we introduce here a certain kind of entropy (in fact a whole family of entropies) for groups which we call local $\mathcal{P}$ entropy. Regarding this entropy are some important comments. First, it is defined not just in terms of the group but also using an additional piece of data in the form of a distinguished infinite cyclic subgroup of the group. Second, it is defined with respect to some chosen class $\mathcal{P}$ of finite groups which is closed under isomorphism. While this choice allows for some flexibility, as a downside, depending on the class of finite groups, the group, and the distinguished cyclic subgroup in question, the entropy may or may not exist. When it does exist, it satisfies some useful basic properties. Among these, one of the key properties is that it is an invariant of the group and choice of distinguished cyclic subgroup: if there is an isomorphism between two groups taking the distinguished cyclic subgroup isomorphically onto the other, then they must each have the same entropy.\\
\indent To give a brief idea of what this local $\mathcal{P}$ entropy attempts to capture, suppose $G$ is a group with some distinguished element $g \in G$ of infinite order, and consider the map $T_{g} \colon G \to G$ given by conjugation by $g$, so $T_{g}(h) = g^{-1}hg$. To study the behavior of $T_{g}$, we may try to consider the growth rate of the fixed point subgroups $\textnormal{Fix}(T_{g}^{n}) = C(g^{n})$, where $C(g^{n})$ denotes the centralizer of $g^{n}$, but are faced with the problem that the $C(g^{n})$'s may be infinite. To proceed, given a class $\mathcal{P}$ of finite groups which is closed under isomorphism, we approximate each of the $C(g^{n})$'s with finite $\mathcal{P}$ groups $H_{n}$, and consider the doubly exponential\footnote{The choice of doubly exponential as opposed to just exponential is motivated by our applications; see Remark \ref{remark:altentropy}.} growth rate of the $H_{n}$'s. The local $\mathcal{P}$ entropy of $(G,g)$ is then the supremum of these growth rates over all such $\mathcal{P}$-approximations. \\
\indent Our motivation here is primarily the use of this entropy as an invariant, and the main application is its use in studying the stabilized automorphism groups of shifts of finite type. For these groups, we use the entropy defined by the class of finite groups which are products of finite simple non-abelian groups satisfying a certain growth condition.\\
\indent The first main result is that for the stabilized automorphism group of a non-trivial mixing shift of finite type, we can recover the topological entropy of the shift of finite type from the local $\mathcal{P}$ entropy of the stabilized group with a distinguished element.

\begin{theorem}\label{thm:entropycalcintro}
Let $(X,\sigma_{X})$ be a non-trivial mixing shift of finite type. For a certain class $\mathcal{C}_{\sigma}$ of finite groups (depending on $(X,\sigma_{X})$) and any $k \ge 1$, the local $\mathcal{C}_{\sigma}$ entropy of $\left(\aut^{(\infty)}(\sigma_{X}),\sigma_{X}^{k}\right)$ is given by
$$h_{\mathcal{C}_{\sigma}}\left(\aut^{(\infty)}(\sigma_{X}),\sigma_{X}^{k}\right) = h_{top}(\sigma_{X}^{k}).$$

\end{theorem}

In the case of full shifts, we can let $\mathcal{C}_{\sigma}$ in Theorem \ref{thm:entropycalcintro} be the class of all finite simple groups.\\
\indent We also introduce a group-theoretic tool which we call ghost centers. These are essentially subgroups of a group which behave like the center with respect to the collection of finitely generated subgroups of the given group. Using this tool together with Theorem \ref{thm:entropycalcintro}, we prove the following.
\begin{theorem}\label{thm:isoentropyintroduction}
Let $(X,\sigma_{X})$, $(Y,\sigma_{Y})$ be non-trivial mixing shifts of finite type, and suppose there is an isomorphism of stabilized automorphism groups
$$\aut^{(\infty)}(\sigma_{X}) \cong \aut^{(\infty)}(\sigma_{Y}).$$
Then
$$\frac{h_{top}(\sigma_{X})}{h_{top}(\sigma_{Y})} \in \mathbb{Q}.$$
\end{theorem}

For a mixing shift of finite type $(X_{A},\sigma_{A})$ presented as the edge shift on some graph with primitive adjacency matrix $A$, the topological entropy is given by $h_{top}(\sigma_{A}) = \log \lambda_{A}$ where $\lambda_{A}$ is the Perron-Frobenius eigenvalue of the matrix $A$. Theorem \ref{thm:isoentropyintroduction} then allows us to give a complete classification of the stabilized automorphism groups of full shifts.

\begin{corollary}
Given natural numbers $m,n \ge 2$, the stabilized groups $\autinf{m}$ and $\autinf{n}$ are isomorphic if and only if $m^{k} = n^{j}$ for some $k,j \in \mathbb{N}$.
\end{corollary}

In particular, we have the following.

\begin{corollary}
The stabilized automorphism groups for the full 2-shift and the full 3-shift are not isomorphic.
\end{corollary}

Finally, using results of Kopra, we also show how the proofs of the above results can be extended to the case of sofic shifts.
\begin{theorem}\label{thm:isoentropysoficintro}
Let $(X,\sigma_{X})$ and $(Y,\sigma_{Y})$ be non-trivial mixing sofic shifts, and suppose there is an isomorphism of stabilized groups
$$\Psi \colon \aut^{(\infty)}(\sigma_{X}) \to \aut^{(\infty)}(\sigma_{Y}).$$
Then
$$\frac{h_{top}(\sigma_{X})}{h_{top}(\sigma_{Y})} \in \mathbb{Q}.$$
\end{theorem}

A brief outline of the paper is as follows. We begin with some background and definitions, and then introduce local $\mathcal{P}$ entropy in Section \ref{sec:localPentropy}. In Section \ref{sec:PSentropycalc} we focus on the case of stabilized automorphism groups of shifts of finite type, and prove Theorem \ref{thm:entropycalcintro}. Section \ref{sec:PSentropycalc} also contains the material on ghost centers, allowing us to prove Theorem \ref{thm:isoentropyintroduction}. In Section \ref{sec:soficstuff} we outline how the proofs may be adapted to sofic shifts, and prove Theorem \ref{thm:isoentropysoficintro}.\\

The author would like to thank Bryna Kra for several helpful conversations related to the contents of this paper, and is very grateful to Mike Boyle for numerous suggestions which improved the paper.

\section{Definitions and background}\label{sec:defsbackground}
By a dynamical system $(X,T)$ we mean a compact metric space $X$ together with a homeomorphism $T \colon X \to X$. An automorphism of a dynamical system $(X,T)$ is a homeomorphism $\phi \colon X \to X$ such that $\phi T = T \phi$, and the collection of all homeomorphisms of $(X,T)$ forms a group (under composition) which we denote by $\aut(T)$. The stabilized automorphism group of $(X,T)$ is defined to be
$$\aut^{(\infty)}(T) = \bigcup_{k=1}^{\infty}\aut(T^{k})$$
where the union is taken in the group $\textnormal{Homeo}(X)$ of all self-homeomorphism of $X$. In other words, $\aut^{(\infty)}(T)$ is precisely the union of the centralizers of all powers of $T$ in $\textnormal{Homeo}(X)$.\\
\indent The following shows how stabilized groups behave with respect to topological conjugacy.
\begin{proposition}\label{prop:ratconjugate}
Let $(X,T), (Y,S)$ be dynamical systems. If there exists $k,j \ge 1$ such that $(X,T^{k})$ and $(Y,S^{j})$ are topologically conjugate, then $\aut^{(\infty)}(T)$ and $\aut^{(\infty)}(S)$ are isomorphic groups.
\end{proposition}
\begin{proof}
If $F \colon (X,T^{k}) \to (Y,S^{j})$ is a topological conjugacy, then
\begin{equation*}
\begin{gathered}
F_{*} \colon \aut^{(\infty)}(T^{k}) \to \aut^{(\infty)}(S^{j})\\
F_{*}(\alpha) = F \alpha F^{-1}
\end{gathered}
\end{equation*}
is an isomorphism. Now observe that $\aut^{(\infty)}(T^{k}) = \aut^{(\infty)}(T)$ and $\aut^{(\infty)}(S^{j}) = \aut^{(\infty)}(S)$.
\end{proof}
Given a finite set $\mathcal{A}$ (referred to as an alphabet) with the discrete topology, the set $\mathcal{A}^{\mathbb{Z}}$ with the product topology is a compact metrizable space. We may consider a point $x \in \mathcal{A}^{\mathbb{Z}}$ as a bi-infinite sequence $x = (x_{i})_{i \in \mathbb{Z}}$ where $x_{i} \in \mathcal{A}$ for all $i$, and the shift map $\sigma \colon \mathcal{A}^{\mathbb{Z}} \to \mathcal{A}^{\mathbb{Z}}$ defined by $\left(\sigma(x)\right)_{i} = x_{i+1}$ is a self-homeomorphism of $\mathcal{A}^{\mathbb{Z}}$. By a \emph{subshift} we mean a system $(X,\sigma_{X})$ where $X$ is a closed subset of $\mathcal{A}^{\mathbb{Z}}$ for some $\mathcal{A}$ and $\sigma_{X}$ is the restriction of the shift map to $X$.\\
\indent Given $(X,\sigma_{X})$ with $X \subset \mathcal{A}^{\mathbb{Z}}$ we define $\mathcal{L}_{m}(X)$ to be the set of words $w = w_{1}\ldots w_{m}$ of length $m$ over $\mathcal{A}$ such that there exists $x \in X$ with $x_{i} = w_{i}$ for all $1 \le i \le m$. The \emph{language} of $X$ is defined to be $\mathcal{L}(X) = \bigcup_{m=1}^{\infty}\mathcal{L}_{m}(X)$.\\
\indent We say $(X,\sigma_{X})$ is \emph{mixing} if for every $u,v \in \mathcal{L}(X)$ there exists $M$ such that for all $m \ge M$, there exists $w \in \mathcal{L}_{m}(X)$ such that $uwv \in \mathcal{L}(X)$.\\
\indent A subshift $(X,\sigma_{X})$ with $X \subset \mathcal{A}^{\mathbb{Z}}$ is called a \emph{shift of finite type} if there is a finite set of words $\mathcal{F}$ defined over $\mathcal{A}$ such that $x \in X$ if and only if no finite substring of $x$ contains a word in $\mathcal{F}$. We call a mixing shift of finite type trivial if it consists of only one finite orbit. \\
\indent When $|\mathcal{A}| = n$ the system $(\mathcal{A}^{\mathbb{Z}},\sigma)$ is called the full $n$-shift, and we denote it by $(X_{n},\sigma_{n})$.\\
\indent For a set $X$, we let $\sym(X)$ denote the group of bijections from $X$ to itself.
\section{Local $\mathcal{P}$ entropy}\label{sec:localPentropy}
By a \emph{leveled group} $(G,g)$ we mean a group $G$ together with a distinguished element $g \in G$ such that the subgroup generated by $g$ in $G$ is infinite cyclic. A \emph{homomorphism of leveled groups} $\phi \colon (G,g) \to (H,h)$ is a group homomorphism $\phi \colon G \to H$ such that $\phi(g) = h$. An isomorphism (resp. monomorphism, epimorphism) of leveled groups $\phi \colon (G,g) \to (H,h)$ is a homomorphism of leveled groups such that $\phi \colon G \to H$ is a group isomorphism (resp. monomorphism, epimorphism).

Let $\mathcal{P}$ be a class of finite groups which is closed under isomorphism. For example, $\mathcal{P}$ could be the collection of finite abelian groups, or finite simple groups. Given a leveled group $(G,g)$, a subgroup $H \subset G$ is called \emph{$g$-locally $\mathcal{P}$} if all of the following hold:
\begin{enumerate}
\item
For all $j \ge 1$, $H \cap C(g^{j})$ is finite.
\item
There exists $J \ge 1$ such that for all $j \ge J$, $H \cap C(g^{j})$ belongs to $\mathcal{P}$.
\item
$H \cap C(g) \ne \{e\}$.
\end{enumerate}

Given $(G,g)$ and $\mathcal{P}$, for $G$ to contain a $g$-locally $\mathcal{P}$ subgroup it is sufficient for $C(g)$ to contain a non-trivial finite group belonging to $\mathcal{P}$.

\begin{definition}\label{def:localPentropy}
Let $\mathcal{P}$ be a class of finite groups which is closed under isomorphism, and let $(G,g)$ be a leveled group such that $G$ contains a $g$-locally $\mathcal{P}$ subgroup. The local $\mathcal{P}$ entropy of $(G,g)$ is defined to be
\begin{equation}\label{eqn:localpentropydef}
h_{\mathcal{P}}(G,g) = \sup_{H \in \mathcal{F}_{\mathcal{P}}} \limsup_{n \to \infty} \frac{1}{n} \log \log |H \cap C(g^{n})|
\end{equation}
where $\mathcal{F}_{\mathcal{P}}$ denotes the set of $g$-locally $\mathcal{P}$ subgroups of $(G,g)$.
\end{definition}

Depending on the group $G$, the subgroup $\langle g \rangle$ generated by $g$, and the class $\mathcal{P}$, the local $\mathcal{P}$ entropy of $(G,g)$ may or may not be defined. For example, consider the leveled abelian group $(\mathbb{Z},1)$. If $\mathcal{P}$ is the class of all finite groups, then $\mathbb{Z}$ contains no $1$-locally $\mathcal{P}$ subgroups, since $\mathbb{Z}$ contains no non-trivial finite group.
\begin{remark}\label{remark:altentropy}
One may define an alternative version of $h_{\mathcal{P}}$ given by
\begin{equation}\label{eqn:otherentropy}
\sup_{H \in \mathcal{F}_{\mathcal{P}}} \limsup_{n \to \infty} \frac{1}{n} \log |H \cap C(g^{n})|.
\end{equation}
Whether to consider quantity \eqref{eqn:localpentropydef} or \eqref{eqn:otherentropy} is a matter of whether one expects, given $G,g$ and $\mathcal{P}$, the growth of $g$-locally $\mathcal{P}$ subgroups in $\bigcup_{n=1}^{\infty}C(g^{n})$ to be exponential or doubly exponential. We've chosen here to focus on the case of double exponential growth, motivated by our applications in later sections.\\
\indent As an interesting example to consider for the quantity \eqref{eqn:otherentropy}, let $T_{A} \colon \mathbb{T}^{2} \to \mathbb{T}^{2}$ be the hyperbolic toral automorphism given by multiplication by the matrix $A = \begin{pmatrix} 2 & 1 \\ 1 & 1 \end{pmatrix}$, and let $G$ be the semidirect product $\mathbb{T}^{2} \rtimes \mathbb{Z}$ where $\mathbb{Z} \to \aut(\mathbb{T}^{2})$ by $1 \mapsto T_{A}$. Any element of finite order in the centralizer of $(0,n) \in G$ must be of the form $(x,0)$, and hence correspond to a point of period $n$ for $T_{A}$. Then one can check that, using the class $\mathcal{P}$ of all finite groups, the quantity defined in \eqref{eqn:otherentropy} for the leveled group $(G,(0,1))$ is precisely $h_{top}(T_{A})$, the topological entropy of the map $T_{A}$.\qed
\end{remark}

\begin{remark}\label{remark:sftpentropyex}
Suppose $(X,\sigma_{X})$ is a non-trivial mixing shift of finite type. By \cite[Theorem 2.3]{BLR88}, if $K$ is any finite group, $\aut(\sigma_{X})$ contains an isomorphic copy of $K$. It follows that for any $k \ne 0$ and any class $\mathcal{P}$ of finite groups which is closed under isomorphism, the local $\mathcal{P}$ entropy of $\left(\aut^{(\infty)}(\sigma_{X}),\sigma_{X}^{k} \right)$ exists.\qed
\end{remark}

\indent In Section \ref{sec:PSentropycalc}, we will show that if $\aut^{(\infty)}(\sigma_{X})$ is the stabilized automorphism group of a non-trivial mixing shift of finite type and $k \in \mathbb{N}$, then for certain classes $\mathcal{P}$ the local $\mathcal{P}$ entropy of $\left( \aut^{(\infty)}(\sigma_{X}),\sigma_{X}^{k} \right)$ is finite and non-zero. We give now a more elementary example.
\begin{example}
Let $p_{k}$ denote the set of primes, $2 \le a \in \mathbb{N}$, and $P$ be a partition of $\mathbb{N}$ such that each piece of the partition has size $p_{k}$ and there are precisely $a^{p_{k}}$ sets in the partition of size $p_{k}$, $k \ge 1$. Let $\tau$ be an element of $\sym(\mathbb{N})$ which acts by a cycle of length $p_{k}$ on each set in the partition of size $p_{k}$.\\
\indent Denote by $\sym(P)$ the subgroup of $\sym(\mathbb{N})$ of finitely supported permutations of $\mathbb{N}$ which respect the partition $P$. Since $\tau^{-1} \sym(P) \tau = \sym(P)$, the subgroup $G$ in $\sym(\mathbb{N})$ generated by $\sym(P)$ and $\tau$ is isomorphic to a semidirect product $\sym(P) \rtimes \mathbb{Z}$ where the copy of $\mathbb{Z}$ is generated by $\tau$.\\
\indent Let $q$ be any prime, and let $\mathcal{P}_{q}$ denote the class of finite elementary abelian $q$-groups (so if $K \in \mathcal{P}_{q}$ then every element of $K$ has order $q$). Then $h_{\mathcal{P}_{q}}\left( G, \tau \right)= \log a$. We give only an outline of the proof, and leave the details to the reader.\\
\indent If $\rho \in C(\tau^{j})$ for some $j$ and is finite order, then $\rho$ must lie in $\sym(P)$ and commute with $\tau^{j}$. Thus if $H$ is any $\tau$-locally $\mathcal{P}_{q}$ subgroup of $G$, any element $\rho \in H \cap C(\tau^{j})$ must be supported on the union of the elements of the partition of size $p$ where $p=q$ or $p \mid j$. Using this one gets that $h_{\mathcal{P}_{q}}\left( G, \tau \right) \le \log a$. Now for each $k \in \mathbb{N}$, let $f(k)$ be the integer floor of $\frac{p_{k} a^{p_{k}}}{q}$, so $p_{k}a^{p_{k}} = q \cdot f(k) + r(k)$ where $r(k) < q$. Then $f(k)$ $q$-cycles fit in the union of the partition elements of size $p_{k}$, so we can define an embedding
\begin{equation*}
\bigoplus_{k=1}^{\infty}H_{k} \hookrightarrow \sym(P), \qquad H_{k} = \bigoplus_{l=1}^{f(k)} \mathbb{Z}/q \mathbb{Z}.
\end{equation*}
Letting $H$ denote the image of this embedding, the subgroup $H$ is a $\tau$-locally $\mathcal{P}_{q}$ subgroup with $\limsup_{j \to \infty} \frac{1}{j}\log \log |H \cap C(\tau^{j})| \ge \log(a)$.
\qed
\end{example}

When the local $\mathcal{P}$ entropy does exist, it satisfies some useful properties.
\begin{proposition}\label{prop:entropyproperties}
Let $\mathcal{P}$ be a class of finite groups which is closed under isomorphism, and let $(G,g)$ and $(H,h)$ be leveled groups. Suppose the local $\mathcal{P}$ entropy for both $(G,g)$ and $(H,h)$ exists.
\begin{enumerate}
\item
Suppose $i \colon (H,h) \to (G,g)$ is a leveled monomorphism. Then
$$h_{\mathcal{P}}(H,h) \le h_{\mathcal{P}}(G,g).$$
\item
Suppose $\Psi \colon (H,h) \to (G,g)$ is a leveled isomorphism. Then
$$h_{\mathcal{P}}(H,g) = h_{\mathcal{P}}(G,g).$$
\item
The local $\mathcal{P}$ entropy of $(G,g^{-1})$ exists, and $h_{\mathcal{P}}(G,g) = h_{\mathcal{P}}(G,g^{-1})$
\end{enumerate}
\end{proposition}

\begin{proof}
For $(1)$, suppose $K \subset H$ is an $h$-local $\mathcal{P}$ subgroup of $H$. Since $i(g) = h$, $i(K \cap C(h^{j})) = i(K) \cap C(g^{j})$, so $i(K)$ is also a $g$-local $\mathcal{P}$ subgroup of $G$. Moreover, for any $j$ we have $|i(K) \cap C(g^{j})| = |K \cap C(h^{j})|$, and it follows that
$$h_{\mathcal{P}}(H,h) \le h_{\mathcal{P}}(G,g).$$
Part $(2)$ follows from $(1)$.
For $(3)$, since $C(g) = C(g^{-1})$, $H$ is a $g$-locally $\mathcal{P}$ subgroup of $G$ if and only if $H$ is a $g^{-1}$-locally $\mathcal{P}$ subgroup of $G$, and $H \cap C(g^{n}) = H \cap C(g^{-n})$. From this it follows that
$$h_{\mathcal{P}}(G,g) = h_{\mathcal{P}}(G,g^{-1}).$$
\end{proof}

\begin{remark}
We do not know whether local $\mathcal{P}$ entropy can increase under leveled factors. \qed
\end{remark}

\subsection{Local product-simple entropy}
In this section we consider local $\mathcal{P}$ entropy where $\mathcal{P}$ is the following class of finite groups.

\begin{definition}\label{def:PSCDRentropy}
Let $C \le 1 \le D$ be positive real numbers and $r \in \mathbb{N}$. Say a finite group $G$ belongs to the class $PS_{C,D,r}$ if there exists finite, simple, non-abelian groups $G_{i}$, $1 \le i \le r$, such that both of the following hold:
\begin{enumerate}
\item
There is an isomorphism $G \cong \prod_{i=1}^{r}G_{i}$.
\item
For all $1 \le i,j \le r$, $|G_{i}|^{C} \le |G_{j}| \le |G_{i}|^{D}$.
\end{enumerate}
\end{definition}
Note the bounds in condition $(2)$ are equivalent to $C \le \frac{\log|G_{j}|}{\log|G_{i}|} \le D$.\\

By definition, for any positive numbers $C \le 1 \le D$ and $r \in \mathbb{N}$, the class $PS_{C,D,r}$ is closed under isomorphism. As a consequence of the Krull-Remak-Schmidt Theorem, if $r_{1} \ne r_{2}$ are natural numbers, then $PS_{C,D,r_{1}}$ and $PS_{C,D,r_{2}}$ are disjoint.\\
\indent When $r=1$, for any positive $C \le 1 \le D$ the class $PS_{C,D,1}$ consists precisely of the finite simple non-abelian groups.\\

We denote the local $PS_{C,D,r}$ entropy of a leveled group $(G,g)$ by $h_{PS_{C,D,r}}(G,g)$.\\

Given $C,D,r$, the $PS_{C,D,r}$ entropy of a leveled group $(G,g)$ may or may not be defined. We will see in Section \ref{sec:PSentropycalc} that for any leveled stabilized automorphism group $\left(\aut^{(\infty)}(\sigma_{X}),\sigma_{X}^{k} \right)$ of a non-trivial mixing shift of finite type $(X,\sigma_{X})$, there are constants $C,D,r$ such that the $PS_{C,D,r}$ entropy of $\left(\aut^{(\infty)}(\sigma_{X}),\sigma_{X}^{k} \right)$ exists, and is finite and positive.\\

The following is the main result of this section, and shows that for leveled stabilized automorphism groups of certain systems, if the $PS_{C,D,r}$ entropy exists, then it is bounded above by the exponential growth rate of the number of periodic points of the system.\\
\indent For a system $(X,T)$ and $m \ge 1$, let $P_{m}(T) = \textnormal{Fix}(T^{m})$ denote the set of $T$-periodic points of period $m$, and let $p_{m}(T) = |P_{m}(T)|$. For a system $(X,T)$ such that $p_{m}(T) < \infty$ for every $m \ge 1$, we define
$$\rho(T) = \limsup_{m \to \infty} \frac{1}{m} \log \left( \max\{p_{m}(T),1\} \right).$$
\begin{theorem}\label{thm:upperboundentgen}
Let $C \le 1 \le D$ be positive real numbers, $r \in \mathbb{N}$. Let $(X,T)$ be a dynamical system such that $p_{m}(T) < \infty$ for every $m \ge 1$, and for any sequence $a_{s} \to \infty$, the set $\bigcup_{s=1}^{\infty}P_{a_{s}}(T)$ is dense in $X$. If the local $PS_{C,D,r}$ entropy of $\left(\aut^{(\infty)}(T),T \right)$ exists, then for any $j \ge 1$
$$h_{PS_{C,D,r}}\left(\aut^{(\infty)}(T),T^{j} \right) \le \rho(T^{j}).$$
\end{theorem}

\begin{remark}
There are systems $(X,T)$ satisfying the hypotheses of Theorem \ref{thm:upperboundentgen} for which the $PS_{C,D,r}$ entropy of $\left(\aut^{(\infty)}(T),T \right)$ does not exist for any choice of $C,D,r$. For example, consider the hyperbolic toral automorphism $T_{A} \colon \mathbb{T}^{2} \to \mathbb{T}^{2}$ induced by multiplication by $A = \begin{pmatrix} 2 & 1 \\ 1 & 1 \end{pmatrix}$. For any $k \ge 1$, the group $\aut(T_{A}^{k})$ is isomorphic to $\mathbb{Z} \times \mathbb{Z}/2\mathbb{Z}$ (see e.g. \cite[Theorem 1]{BaakeRoberts1997}), so the only finite group contained in $\aut(T_{A}^{k})$ is isomorphic to $\mathbb{Z}/2\mathbb{Z}$. Note however that if one lets $\mathcal{P}_{ab}$ denote the class of finite abelian groups, then $h_{\mathcal{P}_{ab}}\left( \aut^{(\infty)}(T_{A}),T_{A} \right)$ does exist, and is equal to zero (see also Remark \ref{remark:altentropy}).\qed
\end{remark}

Before beginning the proof of Theorem \ref{thm:upperboundentgen}, we record two elementary results from group theory which will be useful for us.

\begin{lemma}\label{lemma:pondlight}
Suppose $G$ is a group, $N$ is a normal subgroup of $G$ and $H$ is a subgroup of $G$ which is simple. If $H \cap N \ne \{e\}$, then $H \subset N$.
\end{lemma}
\begin{proof}
Since $N$ is normal in $G$, $N \cap H$ is normal in $H$. Then $H \cap N \ne \{e\}$ implies $H \cap N$ is a non-trivial normal subgroup of $H$; since $H$ is simple, this means $H \cap N = H$, and $H \subset N$.
\end{proof}

\begin{lemma}\label{lemma:smallisland}
Let $G = \prod_{i=1}^{r}G_{i}$ be a group where each $G_{i}$ is finite, simple, and non-abelian. Let $\pi_{i} \colon G \to G_{i}$ denote the homomorphism given by projection onto the $i$th coordinate of $G$. Suppose $N$ is a normal subgroup of $G$, and let $I_{n} = \{i \mid \pi_{i}(N) \ne \{e\}\} \subset \{1,\ldots,r\}$. Then
$$\prod_{i \in I_{N}}G_{i} \subset N.$$
\end{lemma}
\begin{proof}
Let $j \in I_{N}$. Then there exists $\tilde{a} = (a_{1},\ldots,a_{r}) \in N$ such that $a_{j} \ne \{e\}$. Given $g \in G_{j}$, let $\tilde{g} = (e,\ldots,g,\ldots,e)$ where the $g$ appears in the $j$th spot. Then
$$\tilde{g}^{-1} \tilde{a}  \tilde{g}  = (a_{1},\ldots,g^{-1}a_{j}g,\ldots,a_{r})$$
and hence
$$(e,\ldots,a_{j}^{-1}g^{-1}a_{j}g,\ldots,e) = \tilde{a}^{-1} \tilde{g}^{-1} \tilde{a}  \tilde{g} \in N.$$
It follows that
$$\{e\} \times \cdots \times [\pi_{j}(N),G_{j}] \times \cdots \times \{e\} \subset N.$$
The subgroup $[\pi_{j}(N),G_{j}]$ is normal in $G_{j}$. Since $G_{j}$ is simple and non-abelian, it has trivial center, so $[\pi_{j}(N),G_{j}]$ is non-trivial. Simplicity of $G_{j}$ then implies
$$\{e\} \times \cdots \times G_{j} \times \cdots \times \{e\} \subset N$$
from which the result follows.

\end{proof}

\begin{proof}[Proof of Theorem \ref{thm:upperboundentropy}]
First we show how the result for arbitrary $j \ge 1$ follows from $j=1$ case. For any $j \ge 1$, note that the system $(X,T^{j})$ satisfies the conditions of the theorem. Then assuming the result holds for $j=1$, we have
\begin{equation*}
h_{PS_{C,D,r}}\left(\aut^{(\infty)}(T),T^{j} \right) = h_{PS_{C,D,r}}\left(\aut^{(\infty)}(T^{j}),T^{j} \right) \le \rho(T^{j}).
\end{equation*}
From here on then we suppose $j=1$. Suppose $H$ is a $T$-locally $PS_{C,D,r}$ subgroup of $\aut^{(\infty)}(T)$. We will show that
\begin{equation}
\limsup_{n \to \infty} \frac{1}{n} \log \log |H \cap C(T^{n})| \le \limsup_{n \to \infty} \frac{1}{n} \log p_{n}(T).
\end{equation}
Consider for each $n$ the homomorphism
\begin{equation*}
\begin{gathered}
\rho_{n} \colon \aut(T^{n}) \to \sym(P_{1}(T^{n}))\\
\rho_{n} \colon \alpha \mapsto \alpha|_{P_{1}(T^{n})}
\end{gathered}
\end{equation*}
and note that the hypothesis on the periodic points implies that for any sequence $a_{s} \to \infty$
\begin{equation}\label{eqn:ppintersection}
\bigcap_{r=1}^{\infty} \ker \rho_{a_{s}} = \{e\}.
\end{equation}
We quickly set some notation to be used for the remainder of the proof. For each $n$, we let $H \cap C(T^{n}) = H^{(n)}$. Since $H$ is a $T$-locally $PS_{C,D,r}$ subgroup, there exists $M \in \mathbb{N}$ such that for all $n \ge M$, we have a decomposition
$$H^{(n)} = \prod_{l=1}^{r}H_{l}^{(n)}$$
where each $H_{l}^{(n)}$ is finite simple non-abelian, and the collection $\{H_{l}^{(n)}\}_{l=1}^{r}$ satisfies condition $(2)$ of Definition \ref{def:PSCDRentropy}.

We will use the following lemma.
\begin{lemma}\label{lemma:lemma1}
Suppose for some $m \ge M$ that
$$|H^{(m)}| > \left(p_{m}(T)!\right)^{rD}.$$
Then
$$H^{(m)} \subset \ker \rho_{m}.$$
\end{lemma}
\begin{proof}[Proof of Lemma \ref{lemma:lemma1}]
First we show the hypothesis implies, for all $1 \le l \le r$,
\begin{equation}\label{eqn:lemma1bound}
|H_{l}^{(m)}| > p_{m}(T)!.
\end{equation}
Given
$$|H^{(m)}| > \left(p_{m}(T)!\right)^{rD}$$
we have
$$\log |H^{(m)}| > \log \left(p_{m}(T)!\right)^{rD}.$$
Since $H^{(m)} = \prod_{l=1}^{r}H_{l}^{(m)}$ this gives
$$\log \prod_{l=1}^{r}|H_{l}^{(m)}| > rD \cdot \log \left(p_{m}(T)!\right)$$
so
$$\sum_{l=1}^{r}\log|H_{l}^{(m)}| > rD \cdot \log \left(p_{m}(T)!\right).$$
Then
$$r \cdot \max_{1 \le l \le r} \{\log |H_{l}^{(m)}|\} > rD \cdot \log \left(p_{m}(T)!\right).$$
By (2) of Definition \ref{def:PSCDRentropy}, for any $1 \le l \le r$ we have
$$D \log |H_{l}^{(m)}| \ge \max_{1 \le l \le r}\{\log|H_{l}^{(m)}|\}$$
so
$$rD \log |H_{l}^{(m)}| \ge r \cdot \max_{1 \le l \le r} \{\log |H_{l}^{(m)}|\} > rD \cdot \log \left(p_{m}(T)!\right).$$
Thus
$$\log |H_{l}^{(m)}| > \log \left(p_{m}(T)!\right)$$
so we get
\begin{equation}\label{eqn:lemma1bound2}
|H_{l}^{(m)}| > p_{m}(T)!.
\end{equation}
Consider for each $l$ the subgroup $K_{l}$ of $H^{(m)}$ given by
$$\{e\} \times \cdots \times H_{l}^{(m)} \times \cdots \times \{e\}.$$
Then for any $1 \le l \le r$, $|K_{l}| = |H_{l}^{(m)}|$, so by \eqref{eqn:lemma1bound2}, $|K_{l}| > \left(p_{m}(T)!\right)$. It follows that the intersection $K_{l} \cap \ker \rho_{m}$ must be non-trivial. Consider the subgroup of $H^{(m)}$ given by $Q = \ker \rho_{m} \cap H^{(m)}$. Since $\ker \rho_{m}$ is normal in $\aut(T^{m})$, $Q$ is normal in $H^{(m)}$. From the above, we know that for any $1 \le l \le r$,
$$\{e\} \ne \pi_{l}(Q) \subset H_{l}^{(m)}.$$
Lemma \ref{lemma:smallisland} then implies $Q$ must contain all of $H^{(m)}$, so $H^{(m)} \subset \ker \rho_{m}$ as desired.
\end{proof}

Continuing the proof of Theorem \ref{thm:upperboundentropy}, we claim there exists $N_{1}$ such that for all $n \ge N_{1}$, we have
\begin{equation}\label{eqn:bound1}
|H^{(n)}| = |H \cap C(T^{n})| \le \left(p_{n}(T)! \right)^{rD}.
\end{equation}
To see this, suppose instead there exists an increasing sequence of natural numbers $b_{i} \to \infty$ such that for all $i$
$$|H^{(b_{i})}| > \left(p_{b_{i}}(T)! \right)^{rD}.$$
By Lemma \ref{lemma:lemma1}, for each $i$ sufficiently large
$$H^{(b_{i})} \subset \ker \rho_{b_{i}}.$$
Since $H$ is a $T$-locally $PS_{C,D,r}$ subgroup, there exists some $\{e\} \ne \alpha \in H \cap C(T)$, so $\alpha \in H \cap C(T^{b_{i}}) = H^{(b_{i})}$ for all $i$ sufficiently large, and hence $\alpha \in \ker \rho_{b_{i}}$ for all $i$ sufficiently large. By \eqref{eqn:ppintersection}, this is a contradiction, proving the claim.\\
\indent Note that, given \eqref{eqn:ppintersection} and the fact that $|H^{(n)}| \ge 2$ for $n \ge M$, we may also assume that $p_{n}(T) \ge 2$ for all $n \ge N_{1}$.\\

Now using \eqref{eqn:bound1} and the fact that for any $L \in \mathbb{N}$, $L! \le L^L$, for any $n \ge N_{1}$ we have
\begin{equation*}
\begin{gathered}
\log \log |H \cap C(T^{n})| \le \log \log \left( (p_{n}(T)!)^{rD} \right) = \log \big( rD \cdot \log \left( p_{n}(T)! \right) \big)\\
\le \log \big( rD \cdot \log \left( p_{n}(T)^{p_{n}(T)} \right) \big) = \log \big( rD \cdot p_{n}(T) \cdot \log \left( p_{n}(T) \right) \big)\\
= \log(rD) + \log(p_{n}(T)) + \log \log (p_{n}(T)).
\end{gathered}
\end{equation*}
It follows that
$$\limsup_{n \to \infty}\frac{1}{n} \log \log |H \cap C(T^{n})| \le \limsup_{n \to \infty} \frac{1}{n} \log p_{n}(T) = \rho(T)$$
completing the proof.
\end{proof}

Recall a system $(X,T)$ is \emph{expansive} if there exists $\epsilon > 0$ such that for $x \ne y$ in $X$, there exists $n \in \mathbb{Z}$ such that $d(T^{n}(x),T^{n}(y)) > \epsilon$ (where $d$ denotes the metric on $X$). For expansive systems satisfying the periodic point density hypothesis of Theorem \ref{thm:upperboundentgen}, we can bound the $PS_{C,D,r}$ entropy of the leveled stabilized automorphism group by the topological entropy.

\begin{theorem}\label{thm:upperboundentropy}
Let $C \le 1 \le D$ be positive real numbers, $r \in \mathbb{N}$. Let $(X,T)$ be an expansive dynamical system such that for any sequence $a_{s} \to \infty$, the set $\bigcup_{s=1}^{\infty}P_{a_{s}}(T)$ is dense in $X$. If the local $PS_{C,D,r}$ entropy of $\left(\aut^{(\infty)}(T),T \right)$ exists, then for any $j \ge 1$
$$h_{PS_{C,D,r}}\left(\aut^{(\infty)}(T),T^{j} \right) \le h_{top}(T^{j}).$$
\end{theorem}
\begin{proof}
Since $(X,T)$ is expansive, $p_{n}(T) < \infty$ for every $n$, and $\rho(T) \le h_{top}(T)$ (see e.g. \cite[Prop. 3.2.14]{HasselblattKatokbook}). Then by Theorem \ref{thm:upperboundentgen} we have
$$h_{PS_{C,D,r}}\left(\aut^{(\infty)}(T),T^{j} \right) \le \rho(T^{j}) \le j \cdot \rho(T) \le j \cdot h_{top}(T) = h_{top}(T^{j}).$$

\end{proof}

\section{Local PS entropy for stabilized automorphism groups of shifts of finite type}\label{sec:PSentropycalc}
We turn now to the study of local PS entropy for stabilized automorphism groups of shifts of finite type. Our first goal is to show that, in the case of shifts of finite type, the local PS entropy of the stabilized automorphism group is given by the topological entropy of an appropriate power of the shift. Since non-trivial mixing shifts of finite type satisfy the hypotheses of Theorem \ref{thm:upperboundentropy}, we already have an upper bound in terms of the entropy, and it remains to prove that this is also a lower bound. We start with some notation.

\subsection{Notation and preliminaries}
Given a square matrix $A$ over $\mathbb{Z}_{+}$, let $\Gamma_{A}$ denote a directed graph with adjacency matrix $A$. We denote the set of vertices by $V(\Gamma_{A})$, and for $i,j \in V(\Gamma_{A})$ the set of edges in $\Gamma_{A}$ from $i$ to $j$ by $E_{i,j}(\Gamma_{A})$. The graph $\Gamma_{A}$ defines an edge shift of finite type $(X_{A},\sigma_{A})$ given by all bi-infinite walks on $\Gamma_{A}$, and any shift of finite type is topologically conjugate to such an edge shift of finite type. Without loss of generality, we may always assume $A$ is nondegenerate, i.e. has no zero row or zero column.\\
\indent Recall a $\mathbb{Z}_{+}$-matrix $A$ is primitive if there exists $K$ such that $A^{K}$ has only positive entries. For a nondegenerate $\mathbb{Z}_{+}$-matrix $A$, the shift of finite type $(X_{A},\sigma_{A})$ is mixing if and only if $A$ is primitive.\\
\indent When $A$ is primitive, $h_{top}(\sigma_{A}) = \log \lambda_{A}$ where $\lambda_{A}$ is the Perron-Frobenius eigenvalue of $A$.

By a \emph{simple graph symmetry} of $\Gamma_{A}$ we mean a graph automorphism of $\Gamma_{A}$ which leaves every vertex fixed. The set of simple graph symmetries of $\Gamma_{A}$ forms a group that we identify with $\prod_{i,j \in V(\Gamma_{A})}\sym(E_{i,j}(\Gamma_{A}))$, where $\sym(X)$ denotes the group of permutations of a set $X$.

Any simple graph symmetry $\tau \in \simp(\Gamma_{A})$ induces an automorphism of $(X_{A},\sigma_{A})$ which we denote by $\tilde{\tau}$. The automorphism $\tilde{\tau}$ is defined by a 0-block code
\begin{equation*}
\tilde{\tau} \colon \ldots x_{-2}x_{-1}x_{0}x_{1}x_{2}\ldots \mapsto \ldots \tau(x_{2})\tau(x_{-1})\tau(x_{0})\tau(x_{1})\tau(x_{2})\ldots.
\end{equation*}
We let $\simp(\Gamma_{A})$ denote the subgroup of simple automorphisms in $\aut(\sigma_{A})$ induced by simple graph symmetries of $\Gamma_{A}$. In other words, there is an isomorphism
\begin{equation*}
\begin{gathered}
\mathfrak{S} \colon \prod_{i,j \in V(\Gamma_{A})}\sym(E_{i,j}(\Gamma_{A})) \stackrel{\cong}\longrightarrow \simp(\Gamma_{A})\\
\mathfrak{S} \colon \tau \mapsto \tilde{\tau}.
\end{gathered}
\end{equation*}

We define the subgroup of even simple graph automorphisms in $\simp(\Gamma_{A})$ by

$$\simp_{\ev}(\Gamma_{A}) = \mathfrak{S}\left(\prod_{i,j \in V(\Gamma_{A})}\alt(E_{i,j}(\Gamma_{A})) \right)$$

where $\alt(X)$ always denotes the alternating subgroup of $\sym(X)$.

Let $\Gamma^{(m)}_{A}$ denote a graph which presents the shift $(X_{A},\sigma_{A}^{m})$; thus $\simp(\Gamma^{(m)}_{A}) \subset \aut(\sigma_{A}^{m})$. Note that we may identify $V(\Gamma_{A}^{(m)})$ with $V(\Gamma_{A})$, and we may identify each edge set $E_{i,j}(\Gamma_{A}^{(m)})$ with paths of length $m$ through $\Gamma_{A}$ going from $i$ to $j$. For any $k, m \ge 1$ we have an inclusion map
\begin{equation}
\label{eq:def-im}
    i_{m,k} \colon \simp(\Gamma^{(m)}_{A}) \hookrightarrow \simp(\Gamma^{(km)}_{A}),
    \end{equation}
and this homomorphism agrees with the restriction of the map
$$\aut(\sigma^{m}_{A}) \hookrightarrow \aut(\sigma_{A}^{km})$$
to $\simp(\Gamma^{(m)}_{A})$.\\
As above, for every $k \ge 1$ we have isomorphisms
\begin{equation}\label{eqn:simpksyms}
\begin{gathered}
\mathfrak{S}_{k} \colon \prod_{i,j \in V(\Gamma^{(k)}_{A})}\sym(E_{i,j}(\Gamma^{(k)}_{A})) \stackrel{\cong}\longrightarrow \simp(\Gamma^{(k)}_{A})\\
\mathfrak{S}_{k}^{\ev} \colon \prod_{i,j \in V(\Gamma^{(k)}_{A})}\alt(E_{i,j}(\Gamma^{(k)}_{A})) \stackrel{\cong}\longrightarrow \simp_{\ev}(\Gamma^{(k)}_{A}).
\end{gathered}
\end{equation}
For a proof of the following, see \cite[Prop. 5.1]{HKS}.
\begin{proposition}
For any $k,m \ge 1$, the map $i_{m,k}$ takes $\simp_{\ev}(\Gamma^{(m)}_{A})$ into $\simp_{\ev}(\Gamma^{(km)}_{A})$.
\end{proposition}

We consider the corresponding stabilized groups, defining the subgroups
$$\simpinf{A} = \bigcup_{m=1}^{\infty} \simp(\Gamma^{(m)}_{A}) \subset \autinf{A}$$
and
$$\simpaltinf{A} = \bigcup_{m=1}^{\infty}\simp_{\ev}(\Gamma^{(m)}_{A}) \subset \simpinf{A}.$$
Thus $\alpha \in \autinf{A}$ lies in $\simpinf{A}$ when  $\alpha$ is induced by a simple graph symmetry of $\Gamma^{(m)}_{A}$ for some $m\geq 1$, and $\alpha \in \simpaltinf{A}$ if for some $m\geq 1$, $\alpha$ is induced by a simple graph symmetry of $\Gamma^{(m)}_{A}$ which consists of only even permutations on every edge set for $\Gamma^{(m)}_{A}$.

\begin{lemma}\label{lemma:factorials}
Suppose $a_{k}, b_{k}$ are sequences of positive integers both converging to $\infty$ such that
$$\lim_{k \to \infty} \frac{a_{k}}{b_{k}} = c > 0.$$
Then
$$\lim_{k \to \infty} \frac{\log a_{k}!}{\log b_{k}!} = c.$$
\end{lemma}

\begin{proof}
Using Stirling's Formula we have
$$\frac{\log a_{k}!}{\log b_{k}!} = \frac{a_{k}\log a_{k} - a_{k} + O(\log a_{k})}{b_{k}\log b_{k} - b_{k} + O(\log b_{k})}$$
$$= \frac{a_{k}}{b_{k}} \left(\frac{\log a_{k} - 1 + \frac{1}{a_{k}} O(\log a_{k})}{\log b_{k} - 1 + \frac{1}{b_{k}} O(\log b_{k})} \right)$$
$$ = \frac{a_{k}}{b_{k}} \left(\frac{\log b_{k} + \log(\frac{a_{k}}{b_{k}}) - 1 + \frac{1}{a_{k}} O(\log a_{k})}{\log b_{k} - 1 + \frac{1}{b_{k}} O(\log b_{k})} \right) \longrightarrow c.$$

\end{proof}

Leading toward Theorem \ref{thm:psentropycalc} ahead, our immediate goal is to determine when, for a given shift of finite type $(X_{A},\sigma_{A})$, the subgroups $\simp_{ev}(\Gamma^{(k)}(A))$ belong to $PS_{C,D,r}$ for some $C,D,r$. With that in mind, we make the following definition.
\begin{definition}\label{def:Aadmissible}
Let $A$ be a primitive $\mathbb{Z}_{+}$-matrix. Given positive constants $C \le 1 \le D$ and $r \in \mathbb{N}$, we say $(C,D,r)$ is $A$-admissible if there exists $K \in \mathbb{N}$ such that for all $k \ge K$, $\simp_{\ev}(\Gamma^{(k)}_{A})$ belongs to the class $PS_{C,D,r}$.
\end{definition}

\begin{lemma}\label{lemma:stirling1}
Let $A$ be an $r \times r$ primitive $\mathbb{Z}_{+}$-matrix.
\begin{enumerate}
\item
There exists positive constants $C \le 1 \le D$ such that $(C,D,r^{2})$ is $A$-admissible.
\item
Suppose $B$ is an $s \times s$ primitive $\mathbb{Z}$-matrix and $r \le s$. Then there are positive constants $E \le 1 \le F$, $l \in \mathbb{N}$, and a primitive $s \times s$ $\mathbb{Z}_{+}$-matrix $A^{\prime}$ such that $(X_{A},\sigma_{A}^{l})$ is topologically conjugate to $(X_{A^{\prime}},\sigma_{A^{\prime}})$ and $(E,F,s^{2})$ is both $A^{\prime}$-admissible and $B$-admissible.
\end{enumerate}
\end{lemma}
\begin{proof}
For part $(1)$, since $A$ is primitive there exists some $K_{1}$ such that for all $k \ge K_{1}$, every entry of $A^{k}$ is at at least five. Then for $k \ge K_{1}$ we have the decomposition
$$\mathfrak{S}_{k}^{\ev} \colon \prod_{i,j \in V(\Gamma^{(k)}_{A})}\alt(E_{i,j}(\Gamma^{(k)}_{A})) \stackrel{\cong}\longrightarrow \simp_{\ev}(\Gamma^{(k)}_{A})$$
where $|V(\Gamma^{(k)}_{A})| = r^{2}$, and each $\alt(E_{i,j}(\Gamma^{(k)}_{A}))$ is simple since $|E_{i,j}(\Gamma_{A}^{(k)})| = A^{k}_{i,j} \ge 5$. It suffices then to verify condition (2) in Definition \ref{def:PSCDRentropy}. Let $\lambda_{A}$ denote the Perron-Frobenius eigenvalue of $A$, let $u$ be a positive left eigenvector for $\lambda_{A}$, $v$ be a positive right eigenvector for $\lambda_{A}$, normalized so that $uv = 1$. Note that every entry of both $u$ and $v$ is positive. As a consequence of the Perron-Frobenius Theorem
$$\lim_{k \to \infty} \left(\frac{1}{\lambda_{A}}A\right)^{k} = vu$$
and this convergence happens exponentially fast. In other words, we may write
$$A^{k} = \lambda_{A}^{k} \cdot vu + F^{(k)}$$
where for every $i,j$, $F^{(k)}_{i,j} / \lambda_{A}^{k} \to 0$ as $k \to \infty$. It follows that for any pairs $i,j$ and $p,q$,
\begin{equation}\label{eqn:pfconv1}
\lim_{k \to \infty}\frac{A_{i,j}^{k}}{A_{p,q}^{k}} = \frac{(vu)_{i,j}}{(vu)_{p,q}}.
\end{equation}
\indent Now for any pair $i,j \in V(\Gamma_{A})$ we have
$$\log |\alt(E_{i,j}(\Gamma_{A}^{(k)}))| = \log \left(\frac{1}{2}|E_{i,j}(\Gamma_{A}^{(k)})|!\right) = \log \left(\frac{1}{2} A_{i,j}^{k}!\right).$$
Thus, given another pair $p,q \in V(\Gamma_{A})$ we have
$$\frac{\log |\alt(E_{i,j}(\Gamma_{A}^{(k)}))|}{\log |\alt(E_{p,q}(\Gamma_{A}^{(k)}))|} = \frac{\log \left(\frac{1}{2} A_{i,j}^{k}!\right)}{\log \left(\frac{1}{2} A_{p,q}^{k}!\right)}$$
$$ = \frac{\log \frac{1}{2} + \log \left(A_{i,j}^{k}!\right)}{\log \frac{1}{2} + \log \left(A_{p,q}^{k}!\right)}.$$
By Lemma \ref{lemma:stirling1} it follows that
$$\lim_{k \to \infty} \frac{\log |\alt(E_{i,j}(\Gamma_{A}^{(k)}))|}{\log |\alt(E_{p,q}(\Gamma_{A}^{(k)}))|} = \frac{(vu)_{i,j}}{(vu)_{p,q}}.$$
Since every entry of $vu$ is positive, there must exist some $K_{2} \in \mathbb{N}$ and positive constants $C \le 1 \le D$ such that for all $k \ge K_{2}$, and for any $i,j,p,q \in V(\Gamma_{A})$, we have
$$C \le \frac{\log |\alt(E_{i,j}(\Gamma_{A}^{(k)}))|}{\log |\alt(E_{p,q}(\Gamma_{A}^{(k)}))|} \le D.$$
This completes the proof of part $(1)$.\\
\indent For part $(2)$, using state splittings (see \cite[Sec. 2.4]{LindMarcus1995}), there exists $l \in \mathbb{N}$ and an $s \times s$ primitive $\mathbb{Z}_{+}$-matrix $A^{\prime}$ such that $(X_{A^{\prime}},\sigma_{A^{\prime}})$ is topologically conjugate to $(X_{A^{l}},\sigma_{A^{l}})$ which itself is topologically conjugate to $(X_{A},\sigma_{A}^{l})$. Now apply part $(1)$ to both $A^{\prime}$ and $B$ to get constants $E_{1},F_{1},E_{2},F_{2}$ such that $(E_{1},F_{1},s^{2})$ is $A^{\prime}$-admissible and $(E_{2},F_{2},s^{2})$ is $B$-admissible. Let $E_{3} = \min\{E_{1},E_{2}\}$ and $F_{3} = \max\{F_{1},F_{2}\}$. Then $(E_{3},F_{3},s^{2})$ is both $A^{\prime}$-admissible and $B$-admissible.
\end{proof}

We now have all the tools we need to compute $h_{PS_{C,D,r}}\left(\autinf{A},\sigma_{A}^{k} \right)$ for a mixing shift of finite type $(X_{A},\sigma_{A})$.

\begin{theorem}\label{thm:psentropycalc}
Let $(X_{A},\sigma_{A})$ be a non-trivial mixing shift of finite type, with $A$ an $r \times r$ primitive $\mathbb{Z}_{+}$-matrix. Let $0 \ne k \in \mathbb{Z}$ be such that $A^{|k|}$ contains an entry strictly greater than 2, and let $(C,D,r^{2})$ be $A$-admissible. Then
$$h_{PS_{C,D,r^{2}}}\left( \autinf{A},\sigma_{A}^{k} \right) = h_{top}(\sigma_{A}^{k}) = \log \left( \lambda_{A}^{|k|} \right).$$
\end{theorem}
\begin{proof}
As noted in Remark \ref{remark:sftpentropyex}, the local $PS_{C,D,r^{2}}$ entropy of $\left( \autinf{A},\sigma_{A}^{k} \right)$ exists. We first show that it is enough to prove the result for $k \ge 1$. Indeed, if it holds for all $k \ge 1$, then by (3) of Proposition \ref{prop:entropyproperties}, for any $k < 0$ we have
$$h_{PS_{C,D,r^{2}}}\left( \autinf{A},\sigma_{A}^{k} \right) = h_{PS_{C,D,r^{2}}}\left( \autinf{A},\sigma_{A}^{-k} \right) = h_{top}(\sigma_{A}^{-k}) = h_{top}(\sigma_{A}^{k}).$$
Thus we assume from here on that $k \ge 1$. The system $(X_{A},\sigma_{A})$ satisfies the hypotheses of Theorem \ref{thm:upperboundentropy}, so we already have
$$h_{PS_{C,D,r^{2}}}\left( \autinf{A},\sigma_{A}^{k} \right) \le h_{top}(\sigma_{A}^{k}).$$
It remains to show that
$$h_{PS_{C,D,r^{2}}}\left( \autinf{A},\sigma_{A}^{k} \right) \ge h_{top}(\sigma_{A}^{k}).$$
For this, it suffices to show there exists a $\sigma_{A}^{k}$-locally $PS_{C,D,r^{2}}$ subgroup $H$ of $\autinf{A}$ such that
$$\limsup_{n \to \infty} \frac{1}{n} \log \log |H \cap C(\sigma_{A}^{kn})| \ge h_{top}(\sigma_{A}^{k}).$$
Consider $H = \simpaltinf{A}$. For any $n \ge 1$, we have
$$H \cap C(\sigma_{A}^{kn}) = \simp_{\ev}(\Gamma_{A}^{(kn)}).$$
Note that since $A^{k}_{p,q} \ge 3$ for some $p,q$, $\simp_{\ev}(\Gamma_{A}^{(k)})$ is not the trivial group, and $H \cap C(\sigma^{k}_{A}) \ne \{e\}$. Since $(C,D,r^{2})$ is $A$-admissible, it follows the subgroup $\simpaltinf{A}$ is a $\sigma_{A}^{k}$-locally $PS_{C,D,r^{2}}$ subgroup of $\autinf{A}$. We will show that
$$\limsup_{n \to \infty} \frac{1}{n} \log \log |\simp_{\ev}(\Gamma_{A}^{(kn)})| = \log(\lambda_{A}^{k}) = h_{top}(\sigma_{A}^{k}).$$
For any $n \ge 1$ sufficiently large,
$$\log \log |\simp_{\ev}(\Gamma_{A}^{(kn)})| = \log \log \prod_{i,j \in V(\Gamma_{A})}|\alt(E_{i,j}(\Gamma_{A}^{(kn)})|$$
$$ = \log \log \prod_{i,j \in V(\Gamma_{A})}\frac{1}{2}|E_{i,j}(\Gamma_{A}^{(kn)})|! = \log \log \prod_{i,j \in V(\Gamma_{A})}\frac{1}{2}A_{i,j}^{kn}!$$
$$ = \log \left( \sum_{i,j \in V(\Gamma_{A})} \log \left( \frac{1}{2}A_{i,j}^{kn}! \right) \right) = \log \left( \sum_{i,j \in V(\Gamma_{A})} \log \frac{1}{2} + \log \left( A_{i,j}^{kn}! \right) \right)$$
$$ \ge \log \left( \log \frac{1}{2} + \log \left( A_{1,1}^{kn}! \right) \right)$$
$$ = \log \left( \log (\lambda_{A}^{kn}!) \left( \frac{\log \frac{1}{2}}{\log (\lambda_{A}^{kn}!)} + \frac{\log (A_{1,1}^{kn}!)}{\log (\lambda_{A}^{kn}!)} \right) \right)$$
$$ = \log \log (\lambda_{A}^{kn}!) + \log \left( \frac{\log \frac{1}{2}}{\log (\lambda_{A}^{kn}!)} + \frac{\log (A_{1,1}^{kn}!)}{\log (\lambda_{A}^{kn}!)} \right).$$
By Perron Frobenius, $A_{1,1}^{kn} / \lambda_{A}^{kn} \to (vu)_{1,1} > 0$, so Lemma \ref{lemma:factorials} implies $\log(A_{1,1}^{kn}!) / \log(\lambda_{A}^{kn}!) \to (vu)_{1,1}$. Thus
$$\lim_{n \to \infty} \frac{1}{n}\log \left( \frac{\log \frac{1}{2}}{\log (\lambda_{A}^{kn}!)} + \frac{\log (A_{1,1}^{kn}!)}{\log (\lambda_{A}^{kn}!)} \right) = 0.$$
For the remaining term, using Stirling's Formula we have
$$\limsup_{n \to \infty} \frac{1}{n} \log \log (\lambda_{A}^{kn}!) = \limsup_{n \to \infty} \frac{1}{n} \log \left( \lambda_{A}^{kn} \log (\lambda_{A}^{kn}) - \lambda_{A}^{kn} +O(\log (\lambda_{A}^{kn})) \right)$$
\begin{equation}\label{eqn:greenbird1}
 = \limsup_{n \to \infty} \frac{1}{n} \left( \log (\lambda_{A}^{kn}) + \log \left( \log(\lambda_{A}^{kn}) - 1 + \frac{1}{\lambda_{A}^{kn}} O( \log (\lambda_{A}^{kn})) \right) \right).
\end{equation}
Since $\log(\lambda_{A}^{kn}) - 1 + \frac{1}{\lambda_{A}^{kn}} O( \log (\lambda_{A}^{kn})) = kn \log \lambda_{A} + - 1 + \frac{1}{\lambda_{A}^{kn}} O( \log (\lambda_{A}^{kn})) = O(n)$,
$$ \lim_{n \to \infty} \frac{1}{n}\log \left( \log(\lambda_{A}^{kn}) - 1 + \frac{1}{\lambda_{A}^{kn}} O( \log (\lambda_{A}^{kn})) \right) = 0.$$
Thus \eqref{eqn:greenbird1} becomes
$$\limsup_{n \to \infty} \frac{1}{n} \log \lambda_{A}^{kn} = \limsup_{n \to \infty} \frac{1}{n} kn \log \lambda_{A} = k \log \lambda_{A}.$$
\end{proof}

The following gives Theorem \ref{thm:entropycalcintro} as stated in the introduction.

\begin{corollary}
Let $(X_{A},\sigma_{A})$ be a non-trivial mixing shift of finite type. For any $0 \ne k \in \mathbb{N}$, there exists positive constants $E \le 1 \le F$ and $s \in \mathbb{N}$ such that
$$h_{PS_{E,F,s^{2}}}\left( \autinf{A},\sigma_{A}^{k} \right) = h_{top}(\sigma_{A}^{k}) = \log \left( \lambda_{A}^{|k|} \right).$$
\end{corollary}

\begin{proof}
Since $(X_{A},\sigma_{A})$ is non-trivial we have $h_{top}(\sigma_{A}) > 0$. Then there exists a primitive $\mathbb{Z}_{+}$-matrix $B$ such that $B$ contains an entry strictly greater than two, and $(X_{A},\sigma_{A})$ and $(X_{B},\sigma_{B})$ are topologically conjugate; this can be seen using the existence of markers (see \cite[Sec. 2]{BLR88}) together with the construction in \cite[App.]{MR970546}. Suppose $B$ is $s \times s$. By part (1) of Lemma \ref{lemma:stirling1} we may choose $E,F,s$ such that $(E,F,s^{2})$ is $B$-admissible. For any $k \ne 0$, there is then a topological conjugacy between $(X_{A},\sigma_{A}^{k})$ and $(X_{B},\sigma_{B}^{k})$ which gives an isomorphism of leveled group $\left( \autinf{A},\sigma_{A}^{k} \right) \cong \left( \autinf{B},\sigma_{B}^{k} \right)$. By part (2) of Proposition \ref{prop:entropyproperties} and Theorem \ref{thm:psentropycalc}, we then have
$$h_{PS_{E,F,s^{2}}}\left( \autinf{A},\sigma_{A}^{k} \right)  =  h_{PS_{E,F,s^{2}}}\left( \autinf{B},\sigma_{B}^{k} \right) = h_{top}(\sigma_{B}^{k}) = h_{top}(\sigma_{A}^{k}).$$
\end{proof}

\begin{remark}
In the case of a full shift $(X_{n},\sigma_{n})$, we may use $C=D=r=1$, in which case $PS_{1,1,1}$ is just the class of finite, simple, non-abelian groups. It turns out if one uses the class $\mathcal{S}$ of all finite simple groups, then $h_{\mathcal{S}}\left(\autinf{A},\sigma_{A}^{k}\right) = h_{top}(\sigma_{A}^{k})$ as well. \qed
\end{remark}

\begin{remark}
Suppose $A$ is a primitive $\mathbb{Z}_{+}$-matrix. Since $\aut(\sigma_{A})$ contains an isomorphic copy of every finite group (see \cite[Sec. 2]{BLR88}), it follows that for any $C,D,r$ and $k \ne 0$, the local $PS_{C,D,r}$ entropy of $\left( \autinf{A},\sigma_{A}^{k} \right)$ exists. We do not always know when this entropy will be positive though. By Lemma \ref{lemma:stirling1} and Theorem \ref{thm:psentropycalc}, there are some constants $C,D$ and $r \in \mathbb{N}$ such that the local $PS_{C,D,r^{2}}$ entropy of $\left(\autinf{A},\sigma_{A}^{j}\right)$ is finite and non-zero. However, we do not know whether, for any arbitrary $A$ and $r \in \mathbb{N}$, there must always exist $j \in \mathbb{N}$ and constants $C,D$ such that the local $PS_{C,D,r}$ entropy of $\left(\autinf{A},\sigma_{A}^{j}\right)$ is non-zero. \qed
\end{remark}

Theorem \ref{thm:psentropycalc} combined with the fact that local $PS_{C,D,r}$ entropy is an invariant of leveled isomorphism leads to the following.

\begin{theorem}\label{thm:levelisotoentropy}
Let $(X_{A},\sigma_{A})$ and $(X_{B},\sigma_{B})$ be non-trivial mixing shifts of finite type. Suppose for some $k, j \ne 0$ there is an isomorphism of leveled groups
$$\Phi \colon \left(\autinf{A},\sigma_{A}^{k}\right) \to \left(\autinf{B},\sigma_{B}^{j}\right).$$
Then
$$\frac{h_{top}(\sigma_{A})}{h_{top}(\sigma_{B})} \in \mathbb{Q}.$$
\end{theorem}
\begin{proof}
Suppose $A$ is $r \times r$, $B$ is $s \times s$ and without loss of generality assume $r \le s$. By Lemma \ref{lemma:stirling1} there exists a primitive $\mathbb{Z}_{+}$-matrix $A^{\prime}$, $l \in \mathbb{N}$, positive constants $E \le 1 \le F$, and a topological conjugacy $\Psi \colon (X_{A},\sigma_{A}^{l}) \to (X_{A^{\prime}},\sigma_{A^{\prime}})$ such that $(E,F,s^{2})$ is both $A^{\prime}$-admissible and $B$-admissible. Given that $A^{\prime}$ and $B$ are both primitive, choose $p \in \mathbb{N}$ such that each of $(A^{\prime})^{p}$ and $B^{p}$ contain an entry strictly greater than two. The conjugacy $\Psi$ induces a leveled isomorphism
$$\Psi_{*} \colon \left(\autinf{A},\sigma_{A}^{l}\right) \to \left(\autinf{A^{\prime}},\sigma_{A^{\prime}}\right)$$
and hence also a leveled isomorphism
$$\Psi_{*} \colon \left(\autinf{A},\sigma_{A}^{lkp}\right) \to \left(\autinf{A^{\prime}},\sigma_{A^{\prime}}^{kp}\right).$$
Using (2) of Proposition \ref{prop:entropyproperties}, the fact that $(E,F,s^{2})$ is $A^{\prime}$-admissible, and Theorem \ref{thm:psentropycalc}, we have
$$h_{PS_{E,F,s^{2}}}\left( \autinf{A},\sigma_{A}^{lkp} \right) = h_{PS_{E,F,s^{2}}}\left(\autinf{A^{\prime}},\sigma_{A^{\prime}}^{kp}\right) = h_{top}(\sigma_{A^{\prime}}^{kp}) =h_{top}(\sigma_{A}^{lkp}).$$
The leveled isomorphism
$$\Phi \colon \left(\autinf{A},\sigma_{A}^{k}\right) \to \left(\autinf{B},\sigma_{B}^{j}\right)$$
also induces a leveled isomorphism
$$\Phi \colon \left(\autinf{A},\sigma_{A}^{lkp}\right) \to \left(\autinf{B},\sigma_{B}^{ljp}\right).$$
Again using (2) of Proposition \ref{prop:entropyproperties}, the fact that $(E,F,s^{2})$ is $B$-admissible, and Theorem \ref{thm:psentropycalc} we get
$$h_{PS_{E,F,s^{2}}}\left( \autinf{A},\sigma_{A}^{lkp} \right) = h_{PS_{E,F,s^{2}}}\left(\autinf{B},\sigma_{B}^{ljp}\right) = h_{top}(\sigma_{B}^{ljp}).$$
Altogether this gives
$$h_{top}(\sigma_{A}^{lkp}) = h_{top}(\sigma_{B}^{ljp})$$
so
$$|lkp| \cdot h_{top}(\sigma_{A}) = h_{top}(\sigma_{A}^{lkp}) = h_{top}(\sigma_{B}^{ljp}) = |ljp| \cdot h_{top}(\sigma_{B})$$
or equivalently,
$$\frac{h_{top}(\sigma_{A})}{h_{top}(\sigma_{B})} = \frac{|j|}{|k|} \in \mathbb{Q}.$$

\end{proof}

\begin{remark}
Another consequence of Theorem \ref{thm:psentropycalc} is the following. Let $(X_{A},\sigma_{A})$ and $(X_{B},\sigma_{B})$ be non-trivial mixing shifts of finite type and suppose there is a monomorphism of leveled stabilized automorphism groups
$$\left(\autinf{A},\sigma_{A}^{k}\right) \hookrightarrow  \left(\autinf{B},\sigma_{B}^{j}\right).$$
Then $h_{top}(\sigma_{A}^{k}) \le h_{top}(\sigma_{B}^{j})$. For example, there can not exist a leveled monomorphism
$$\left(\autinf{3},\sigma_{3}\right) \hookrightarrow  \left(\autinf{2},\sigma_{2}\right).$$
It is worth noting that a stabilized version of the Kim-Roush Embedding Theorem \cite[Theorem 4.2]{HKS} shows that, for any $n \ge 2$ and any mixing shift of finite type $(X_{A},\sigma_{A})$, there is a \emph{group} monomorphism (i.e. not a monomorphism of leveled groups)
$$\autinf{n} \hookrightarrow \autinf{A}.$$
We do not know, for any natural numbers $m < n$ both greater than two, whether there exists a monomorphism of leveled groups
$$\left(\autinf{m},\sigma_{m}\right) \hookrightarrow \left(\autinf{n},\sigma_{n}\right).$$
In fact, we do not know whether there can even exist a monomorphism of leveled non-stabilized automorphism groups
$$\left(\aut(\sigma_{m}),\sigma_{m}\right) \hookrightarrow \left(\aut(\sigma_{n}),\sigma_{n}\right).$$
\qed
\end{remark}

\subsection{Ghost centers}
Given a group $G$, a \emph{ghost center} is a subgroup $H \subset G$ such that $H \cap C(K) \ne \{e\}$ for every finitely generated subgroup $K \subset G$. A \emph{cyclic ghost center} is a subgroup $H$ which is both cyclic and a ghost center, and a \emph{maximal cyclic ghost center} is a cyclic ghost center $H$ such that if $H^{\prime}$ is another cyclic ghost center and $H \subset H^{\prime}$, then $H = H^{\prime}$.\\

\textbf{Example: } For the additive group $\mathbb{Q}$, any cyclic subgroup is of course a cyclic ghost center - but there are no maximal cyclic ghost centers (as there are no maximal cyclic subgroups in $\mathbb{Q}$).

\begin{proposition}\label{prop:isopreservesmcgc}
Suppose $\Psi \colon G_{1} \to G_{2}$ is an isomorphism of groups, and $H \subset G_{1}$ is a maximal cyclic ghost center. Then $\Psi(H)$ is a maximal cyclic ghost center in $G_{2}$.
\end{proposition}
\begin{proof}
Note the group $\Psi(H)$ is cyclic. If $K$ is a finitely generated subgroup of $G_{2}$, then $K^{\prime} = \Psi^{-1}(K)$ is also finitely generated, so there exists $x \ne e \in H \cap C(K^{\prime})$. Then $e \ne \Psi(x) \in \Psi(H) \cap \Psi(C(K^{\prime})) = \Psi(H) \cap C(K)$, so $\Psi(H)$ is a cyclic ghost center. To see that it is maximal, suppose $H^{\prime}$ is a cyclic ghost center in $G_{2}$ with $\Psi(H) \subset H^{\prime}$. Then by the above argument, $\Psi^{-1}(H^{\prime})$ is a cyclic ghost center of $G_{1}$, and contains $H$. Since $H$ is maximal, this implies $H = \Psi^{-1}(H^{\prime})$, or equivalently, $\Psi(H) = H^{\prime}$.
\end{proof}

We'll now show that $\autinf{A}$ has maximal cyclic ghost centers, and characterize them. Define the root set of $\sigma_{A}$ in $\autinf{A}$ to be
$$\textnormal{root}(\sigma_{A}) = \{ \gamma \in \autinf{A} \mid \gamma^{k} = \sigma_{A}^{j} \textnormal{ for some } k,j \ne 0\}.$$
An element $\alpha \in \autinf{A}$ is \emph{rootless} if $\gamma^{k} = \alpha$ implies $|k| = 1$.
\begin{proposition}\label{prop:charcgc}
An automorphism $\gamma$ generates a cyclic ghost center in $\autinf{A}$ if and only if $\gamma \in \textnormal{root}(\sigma_{A})$. Such a cyclic ghost center will be maximal if and only if $\gamma$ is rootless.
\end{proposition}
\begin{proof}
Suppose $\gamma \in \textnormal{root}(\sigma_{A})$, so there exists $k,j \ne 0$ such that $\gamma^{k} = \sigma_{A}^{j}$. Since any finitely generated subgroup $K$ in $\autinf{A}$ is contained in $\autk{m}{A}$ for some $m$, $\gamma^{mk} =\sigma_{A}^{mj} \in C(K)$, and $\langle \gamma \rangle \cap C(K) \ne \{e\}$.\\
\indent Conversely, suppose $\gamma$ generates a cyclic ghost center in $\autinf{A}$, and $\gamma \in \aut(\sigma_{A}^{j})$. By \cite{Kopra2019}, there exists a finitely generated subgroup $K \subset \aut(\sigma_{A}^{j})$ such that $C(K) = \langle \sigma_{A}^{j} \rangle$. By assumption, there exists $k$ such that $\{e\} \ne \gamma^{k} \in C(K)$, so $\gamma^{k} = \sigma_{A}^{js}$ for some $s \ne 0$.\\
\indent The claim regarding maximality is clear, since $\langle \gamma \rangle \subset \langle \gamma^{\prime} \rangle$ if and only if $\gamma = (\gamma^{\prime})^{n}$ for some $n$.
\end{proof}
The following shows that there exist rootless roots of $\sigma_{A}$, and hence maximal cyclic ghost centers in $\autinf{A}$.
\begin{proposition}\label{prop:ghostcentersexist}
For any $j \ne 0$, there exists $k(j) \ge 1$ such that if $\alpha \in \autinf{A}$ satisfies $\alpha^{k} = \sigma_{A}^{j}$, then $|k| \le k(j)$. Thus for any $\gamma \in \textnormal{root}(\sigma_{A})$ there exists a maximal cyclic subgroup containing $\gamma$.
\end{proposition}
\begin{proof}
First suppose $j > 0$. Then the homeomorphism $(X_{A},\sigma_{A}^{j})$ is conjugate to the shift of finite type $(X_{A^{j}},\sigma_{A^{j}})$. By \cite[Thm. 8]{lind84}, a necessary condition for $\sigma_{A^{j}}$ to have a $k$th root is that $\left( \lambda_{A^{j}} \right)^{1/k} = \left( \lambda_{A}^{j} \right)^{1/k}$ is a Perron number. By \cite[Thm. 4]{lind84}, there are only finitely many such $k$.\\
If instead $j < 0$, then letting $B = A^{T}$ (where $A^{T}$ denotes the transpose of $A$), $(X_{A},\sigma_{A}^{j})$ is conjugate to $(X_{B^{|j|}},\sigma_{B^{|j|}})$, and the above argument applies.
\end{proof}

\begin{remark}
The finitary Ryan's Theorem used in Proposition \ref{prop:charcgc} also holds for sofic shifts (see \cite[Theorem 5.9]{Kopra2020}), but does not extend to the class of all subshifts with specification, by \cite[Theorem 6.11]{Kopra2020}.\qed
\end{remark}

\subsection{Distinguishing stabilized automorphism groups}
 The goal is to now show how Theorem \ref{thm:isoentropyintroduction} follows from Theorem \ref{thm:levelisotoentropy}. The main idea is to use ghost centers to prove that any isomorphism of stabilized automorphism groups of non-trivial mixing shifts of finite type is also an isomorphism of leveled groups.
\begin{lemma}\label{lemma:isoupgrade}
If $\Psi \colon \autinf{A} \to \autinf{B}$ is an isomorphism, then for some $k,j \ne 0$, $\Psi(\sigma_{A}^{k}) = \sigma_{B}^{j}$. In particular, $\Psi$ is also an isomorphism of leveled groups $\Psi \colon \left(\autinf{A},\sigma_{A}^{k} \right) \to \left(\autinf{B},\sigma_{B}^{j} \right)$.
\end{lemma}
\begin{proof}
By Propositions \ref{prop:ghostcentersexist} and \ref{prop:charcgc}, we may choose a maximal cyclic ghost center $H = \langle \gamma \rangle$ of $\autinf{A}$ where $\gamma^{a} = \sigma_{A}^{b}$ for some $a,b \ne 0$. By Proposition \ref{prop:isopreservesmcgc}, $\Psi(H)$ is a maximal cyclic ghost center of $\autinf{B}$, so there exists some $c,d \ne 0$ such that $\left(\Psi(\gamma)\right)^{c} = \sigma_{B}^{d}$. Then $\Psi(\sigma_{A}^{bc}) = \Psi(\gamma^{ac}) = (\Psi(\gamma))^{ac} = \sigma_{B}^{da}$.
\end{proof}

In general, an isomorphism $\Psi \colon \autinf{A} \to \autinf{B}$ need not satisfy $\Psi(\sigma_{A}) = \sigma_{B}$. For example, upon choosing a topological conjugacy $F \colon (X_{4},\sigma_{4}) \to (X_{2},\sigma_{2}^{2})$, the induced isomorphism $F_{*} \colon \autinf{4} \to \autinf{2}$ satisfies $F_{*}(\sigma_{4}) = \sigma_{2}^{2}$.\\

We now prove Theorem \ref{thm:isoentropyintroduction} from the introduction.

\begin{theorem}\label{thm:isotoentropy}
Let $(X_{A},\sigma_{A})$, $(X_{B},\sigma_{B})$ be non-trivial mixing shifts of finite type, and suppose there is an isomorphism of stabilized automorphism groups
$$\autinf{A} \cong \autinf{B}.$$
Then
$$\frac{h_{top}(\sigma_{A})}{h_{top}(\sigma_{B})} \in \mathbb{Q}.$$
\end{theorem}
\begin{proof}
Suppose $\Psi \colon \autinf{A} \to \autinf{B}$ is an isomorphism. By Lemma \ref{lemma:isoupgrade}, for some $k,j \ne 0$, $\Psi$ is an isomorphism of leveled groups
$$\Psi \colon \left(\autinf{A},\sigma_{A}^{k} \right) \to \left(\autinf{B},\sigma_{B}^{j} \right).$$
The result then follows from Theorem \ref{thm:levelisotoentropy}.
\end{proof}

As mentioned in the introduction, Theorem \ref{thm:isotoentropy} allows us to classify, up to isomorphism, the stabilized automorphism groups of full shifts.

\begin{corollary}
Given natural numbers $m,n \ge 2$, the stabilized groups $\autinf{m}$ and $\autinf{n}$ are isomorphic if and only if $m^{k} = n^{j}$ for some $k,j \in \mathbb{N}$.
\end{corollary}
\begin{proof}
If $m^{k} = n^{j}$ for some $k,j \in \mathbb{N}$, then since $(X_{m},\sigma_{m}^{k})$ is topologically conjugate to $(X_{m^k},\sigma_{m^{k}})$ and $(X_{n},\sigma_{n}^{j})$ is topologically conjugate to $(X_{n^{j}},\sigma_{n^{j}})$, Proposition \ref{prop:ratconjugate} implies $\autinf{m}$ and $\autinf{n}$ are isomorphic. For the other direction, suppose $\autinf{m}$ and $\autinf{n}$ are isomorphic. Since $h_{top}(\sigma_{m}) = \log m$ and $h_{top}(\sigma_{n}) = \log n$, Theorem \ref{thm:isotoentropy} implies $\frac{\log m}{\log n} = \frac{j}{k}$ for some $k,j \in \mathbb{N}$, and hence $m^{k} = n^{j}$.
\end{proof}

As another example, consider $A = \begin{pmatrix} 1 & 1 \\ 1 & 0 \end{pmatrix}$. The shift of finite type $(X_{A},\sigma_{A})$ is a topologically conjugate to the `golden mean' shift. Since $h_{top}(\sigma_{A}) = \log \left(\frac{1 + \sqrt{5}}{2} \right)$ and $\left(\frac{1 + \sqrt{5}}{2}\right)^{k}$ is never an integer, it follows from Theorem \ref{thm:isotoentropy} that the stabilized group $\autinf{A}$ is not isomorphic to the stabilized automorphism group of any full shift.

\section{Sofic shifts}\label{sec:soficstuff}
In this section, we will show how the proofs of our previous results can be adapted to the setting of mixing sofic shifts, giving the following.
\begin{theorem}
Let $(X,\sigma_{X})$ and $(Y,\sigma_{Y})$ be non-trivial mixing sofic shifts, and suppose there is an isomorphism of stabilized groups
$$\Psi \colon \aut^{(\infty)}(\sigma_{X}) \to \aut^{(\infty)}(\sigma_{Y}).$$
Then
$$\frac{h_{top}(\sigma_{X})}{h_{top}(\sigma_{Y})} \in \mathbb{Q}.$$

\end{theorem}

The proof proceeds along the same lines as that of Theorem \ref{thm:isotoentropy}, with a few details needed in some areas. First we show that Lemma \ref{lemma:isoupgrade} also applies to sofic shifts. By \cite[Theorem 5.9]{Kopra2020}, if $(Z,\sigma_{Z})$ is a non-trivial mixing sofic shift, then there exists a finitely generated subgroup $K$ in $\aut(\sigma_{X})$ such that the centralizer of $K$ is precisely $\langle \sigma_{X} \rangle$. The proofs of Propositions \ref{prop:ghostcentersexist} and \ref{prop:charcgc} then go through with appropriate modifications to the setting of sofic shifts, giving Lemma \ref{lemma:isoupgrade}. It suffices then to show that Theorem \ref{thm:levelisotoentropy} holds for non-trivial mixing sofic shifts. \\
\indent If $(X,\sigma_{X}), (Y,\sigma_{Y})$ are non-trivial mixing sofic shifts, we may choose right-resolving presentations $(\Gamma_{X},\mathcal{L}_{X}), (\Gamma_{Y},\mathcal{L}_{Y})$ respectively, such that the adjacency matrices $A(\Gamma_{X})$ and $B(\Gamma_{Y})$ are each primitive. Without loss of generality, assume that $|V(\Gamma_{X})| \le |V(\Gamma_{Y})|$. Since the in-splitting of a right-resolving labeled graph is right-resolving, using in-splittings (see e.g.
\cite[Sec. 2.4]{LindMarcus1995}), there exists $l \in \mathbb{N}$ and a right-resolving presentation $(\Gamma^{\prime}_{X},\mathcal{L}^{\prime}_{X})$ of $(X,\sigma_{X}^{l})$ such that $|V(\Gamma^{\prime}_{X})| = |V(\Gamma_{Y})|$ and the adjacency matrix $A(\Gamma^{\prime}_{X})$ is primitive. The proof of Theorem \ref{thm:levelisotoentropy} then goes through, once we prove the following lemma.
\begin{lemma}
Suppose $(Z,\sigma_{Z})$ is a non-trivial mixing sofic shift with a right-resolving presentation $(\Gamma_{Z},\mathcal{L}_{Z})$ of $(Z,\sigma_{Z})$ such that the adjacency matrix $A_{\Gamma_{Z}}$ is primitive. Suppose $|V(\Gamma_{Z})| = r$, and $L \in \mathbb{Z}$ is such that each entry of $A_{\Gamma_{Z}}^{L}$ is at least $3r$. Then
$$h_{PS_{1,1,r^{2}}} \left(\aut^{(\infty)}(\sigma_{Z}),\sigma_{Z}^{L} \right) = h_{top}(\sigma_{Z}^{L}).$$
\end{lemma}
\begin{proof}
For ease of notation, we let $(\Gamma,\mathcal{L})$ denote the right-resolving presentation of $(Z,\sigma_{Z}^{L})$ given by paths of length $L$ through $(\Gamma_{Z},\mathcal{L}_{Z})$, and let $A$ denote the adjacency matrix of $\Gamma$. Note that for any $t \ge 1$, $\aut(\sigma_{Z})$ contains an isomorphic copy of $\prod_{i=1}^{t}\alt_{5}$ (see \cite{S27}), so for any $C,D,t$ the local $PS_{C,D,t}$ entropy of $\left(\aut^{(\infty)}(\sigma_{Z}),\sigma_{Z}^{L}\right)$ exists. Similar to the proof of Theorem \ref{thm:psentropycalc}, we may assume then that $L > 0$. We may also assume that $r \ge 2$.\\
\indent Since non-trivial mixing sofic shifts satisfy the hypotheses of Theorem \ref{thm:upperboundentropy}, we already have
$$h_{PS_{1,1,r^{2}}} \left(\aut^{(\infty)}(\sigma_{Z}),\sigma_{Z}^{L} \right) \le h_{top}(\sigma_{Z}^{L}).$$

It suffices then to find a $\sigma_{Z}^{L}$-locally $PS_{1,1,r^{2}}$ subgroup $H$ of $\aut^{(\infty)}(\sigma_{Z})$ such that
$$\limsup_{n \to \infty} \frac{1}{n} \log \log |H \cap C(\sigma_{Z}^{nL})| = h_{top}(\sigma_{Z}^{L}).$$
We use an approach similar to that of the proof of Theorem \ref{thm:psentropycalc}. For each $k$ we have a right-resolving presentation $(\Gamma^{(k)},\mathcal{L}^{(k)})$ of $(Z,\sigma_{Z}^{kL})$, where $\mathcal{L}^{(k)}$ is the labeling of $\Gamma^{(k)}$ given by paths of length $k$ through $(\Gamma,\mathcal{L})$. The only difficulty is that we do not have isomorphisms as in \eqref{eqn:simpksyms}, since $\Gamma^{(k)}$ may contain edges with repeated labels, so the maps $\mathfrak{S}_{k}$ defined in \eqref{eqn:simpksyms} may not be injective. To remedy this, we will define a modified version of $\simp^{(\infty)}(\Gamma)$ which will work.\\
\indent For any $j$ dividing $k$, and subgraph $\mathcal{G}^{(j)} \subset \Gamma^{(j)}$, denote by $\iota_{j,k}(\mathcal{G}^{(j)})$ the subgraph of $\Gamma^{(k)}$ given by paths of length $\frac{k}{j}$ through $\mathcal{G}^{(j)}$.\\
\indent First note that for any $j$, since $(\Gamma^{(j)},\mathcal{L}^{(j)})$ is also right-resolving, given $l \in E(\Gamma^{(j)})$ there are at most $r$ edges in $\Gamma^{(j)}$ labeled with $l$, no two of which start at the same vertex. Define $b_{j} = \frac{1}{r}\min_{p,q}A^{j}_{p,q}$. We will inductively define spanning subgraphs $\mathcal{G}^{(j)} \subset \Gamma^{(j)}$ for $j \in \mathbb{N}$ such that all of the following hold:
\begin{enumerate}
\item
$|E_{p,q}(\mathcal{G}^{(j)})| = \floor*{b_{j}}$ for all $p,q \in V(\Gamma)$.
\item
For all $d \mid j$, $\mathcal{G}^{(j)}$ contains $\iota_{d,j}(\mathcal{G}^{(d)})$.
\item
For all $j$, $\mathcal{L}^{(j)}|_{E(\mathcal{G}^{(j)})}$ is injective.
\end{enumerate}

Given this, we can then consider the subgroups
$$\simp_{\ev}(\mathcal{G}^{(j)}) = \{ \tilde{\tau} \mid \tau \in \prod_{p,q \in V(\Gamma)}\alt(E_{p,q}(\mathcal{G}^{(j)}))\}$$
where as before, $\tilde{\tau}$ denotes the automorphism given by the 0-block code coming from the simple graph symmetry $\tau$. Using condition $(2)$, for any $d \mid j$ the image of $\simp_{\ev}(\mathcal{G}^{(d)})$ under the inclusion map $\aut(\sigma_{Z}^{dL}) \hookrightarrow \aut(\sigma_{Z}^{jL})$ is contained in $\simp_{\ev}(\mathcal{G}^{(j)})$, and we define
$$\simp_{\ev}^{(\infty)}(\mathcal{G}) = \bigcup_{j=1}^{\infty}\simp_{\ev}(\mathcal{G}^{(j)}).$$
Condition $(3)$ implies for large enough $j$, each $\simp_{\ev}(\mathcal{G}^{(j)})$ is $PS_{1,1,r^{2}}$. Since each entry of $A$ is at least $3r$, $\simp_{\ev}(\mathcal{G}^{(1)})$ is non-trivial, and it follows that $\simp_{\ev}^{(\infty)}(\mathcal{G})$ is $\sigma_{Z}^{L}$-locally $PS_{1,1,r^{2}}$. It remains to check that $\limsup_{n \to \infty} \frac{1}{n} \log \log |\simp_{\ev}^{(\infty)}(\mathcal{G}) \cap C(\sigma_{Z}^{Ln})| = h_{top}(\sigma_{Z}^{L})$. Using condition $(1)$ we have
$$\log \log |\simp_{\ev}^{(\infty)}(\mathcal{G}) \cap C(\sigma_{Z}^{Ln})| = \log \log |\simp_{\ev}(\mathcal{G}^{(n)})|$$
$$ = \log \log \prod_{p,q \in V(\Gamma)}|\alt(E_{p,q}(\mathcal{G}^{(n)}))| = \log \log \prod_{p,q \in V(\Gamma)} \frac{1}{2}\left( E_{p,q}(\mathcal{G}^{(n)}) \right)!$$
$$ = \log \log \prod_{p,q \in V(\Gamma)} \frac{1}{2}\left( \floor*{\frac{1}{r}\min_{i,j}A^{n}_{i,j}} \right)!$$
$$ = \log \left[ \sum_{p,q \in V(\Gamma)} \log(\frac{1}{2}) + \log \left( \floor*{\frac{1}{r}\min_{i,j}A^{n}_{i,j}} \right)! \right]$$
$$ \ge \log \left[ \sum_{p,q \in V(\Gamma)} \log(\frac{1}{2}) + \log \left( \frac{1}{2r}\min_{i,j}A^{n}_{i,j} \right)! \right].$$
For $n$ sufficiently large, this last term is greater than
$$ \log \log \left( \frac{1}{2r}\min_{i,j}A^{n}_{i,j} \right)!.$$
The rest of the calculation then proceeds analogously to the one in Theorem \ref{thm:psentropycalc}; using Stirling's Formula we have
$$\limsup_{n \to \infty} \frac{1}{n} \log \log \left( \frac{1}{2r}\min_{i,j}A^{n}_{i,j} \right)! = \limsup_{n \to \infty} \frac{1}{n} \cdot n \log(\lambda_{A}) $$
$$ = \log(\lambda_{A}) = h_{top}(\sigma_{Z}^{L}).$$
All that remains is to construct the subgraphs $\mathcal{G}^{(j)}$ satisfying the listed properties. We do this inductively.\\
\indent First consider when $j=1$. Since the presentation $(\Gamma,\mathcal{L})$ is right resolving, each edge shares a label with at most $r$ other edges, each of which must start from a different vertex. Then given that each entry of $A$ is at least $3r$, $b_{1} = \frac{1}{r}\min_{p,q} A_{p,q} \ge 3$. It follows we may select subsets $F_{p,q} \subset E_{p,q}(\Gamma)$ such that $|F_{p,q}| = \floor*{\frac{1}{r}\min_{p,q} A_{p,q}}$ for each $p,q$, and $\mathcal{L}$, when restricted to $\bigcup_{p,q}F_{p,q}$, is injective.\\
\indent Now suppose for all $1 \le j \le J-1$ we have $\mathcal{G}^{(j)}$ satisfying $(1)$ though $(3)$ above. We will show how to define $\mathcal{G}^{(J)}$.\\
\indent Write $J = \prod_{i=1}^{s}p_{i}^{a_{i}}$ and let $d_{i} = \frac{J}{p_{i}}$. Fix $p,q \in V(\Gamma)$. For a number $x$, let $M_{x}$ denote the $r \times r$ matrix all of whose entries equal $x$. For each $i$, the number of edges in $E_{p,q}(\Gamma^{(J)})$ coming from $\iota_{d_{i},J}(\mathcal{G}^{(d_{i})})$ is $$\left(M_{\floor*{b_{d_{i}}}}^{p_{i}}\right)_{p,q} \le \left( M_{b_{d_{i}}}^{p_{i}} \right)_{p,q} = \left( \frac{1}{r^{p_{i}}}M^{p_{i}}_{\min_{p,q}A^{d_{i}}_{p,q}} \right)_{p,q}$$
$$ \le \frac{1}{r^{p_{i}}} \min_{p,q}\left(A^{J}\right)_{p,q}.$$

Using the inductive hypothesis, any edge in $E_{p,q}(\Gamma^{(J)})$ contained in $\iota_{d,J}(\mathcal{G}^{(d)})$ for some $d \mid J$ is contained $\iota_{d_{i},J}(\mathcal{G}^{(d_{i})})$ for some $d_{i}$. Thus the total number of edges in $E_{p,q}(\Gamma^{(J)})$ coming from $\iota_{d,J}(\mathcal{G}^{(d)})$ for all $d \mid J$ is bounded above by
\begin{equation*}\label{eqn:soficbound1}
\sum_{i=1}^{s}\frac{1}{r^{p_{i}}}\min_{p,q}\left(A^{J}\right)_{p,q}.
\end{equation*}
Note that $\sum_{i=1}^{s}\frac{1}{r^{p_{i}}} < \sum_{i=1}^{\infty} \frac{1}{r^{p_{i}}} < \sum_{i=1}^{\infty}\frac{1}{r^{3i}} = \frac{1}{r^{3}-1}$, so
\begin{equation}\label{eqn:soficbound2}
\sum_{i=1}^{s}\frac{1}{r^{p_{i}}}\min_{p,q}\left(A^{J}\right)_{p,q} < \left(\frac{1}{r^{3}-1}\right)\min_{p,q}\left(A^{J}\right)_{p,q}.
\end{equation}

Let $\mathcal{D} = \bigcup_{d \mid J}\iota_{d,J}(\mathcal{G}^{(d)})$ and define for each $p,q$
$$R_{p,q} = \{e \in E_{p,q}(\Gamma^{(J)}) \mid \mathcal{L}^{(J)}(e) = \mathcal{L}^{(J)}(f) \textnormal{ for some } f \in E_{u,v}(\mathcal{D}) \textnormal{ and some } u,v\}.$$
By \eqref{eqn:soficbound2} and the fact that $(\Gamma^{(J)},\mathcal{L}^{(J)})$ is right-resolving, $|R_{p,q}| \le \frac{r}{r^{3}-1}\min_{p,q}\left(A^{J}\right)_{p,q}.$

It follows that we may choose, for each $p,q$, subsets $E_{p,q}^{\prime} \subset E_{p,q}(\Gamma^{(J)})$ such that:
\begin{enumerate}
\item
$E_{p,q}(\mathcal{D}) \subset E_{p,q}^{\prime}$.
\item
$|E_{p,q}^{\prime}| = \floor*{\frac{1}{r}\min_{p,q}A^{J}_{p,,q}}$.
\item
$\mathcal{L}^{(J)}$ is injective on $\bigcup_{p,q}E_{p,q}^{\prime}$.
\end{enumerate}
Finally, we can define $\mathcal{G}^{(J)}$ to be the spanning subgraph of $\Gamma^{(J)}$ with edge set $\bigcup_{p,q}E_{p,q}^{\prime}$.

\end{proof}

\bibliographystyle{plain}
\bibliography{ssbib5}

\end{document}